\documentclass{article}
\usepackage[utf8]{inputenc}

\usepackage{amssymb,amsmath,epsfig}
\usepackage{mathtools}
\usepackage{epstopdf}
\usepackage{listings}
\usepackage{xcolor,soul}
\usepackage{bm,xspace}
\usepackage{comment}
\usepackage{xcolor,soul}
\usepackage{graphicx}
\usepackage{url}
\usepackage{float}
\usepackage{hyperref}
\usepackage{lipsum}
\usepackage{cite}

\usepackage{authblk}

\newcommand{\ber}{\begin{eqnarray}}
\newcommand{\eer}{\end{eqnarray}}

\newtheorem{remark}{\noindent Remark}

\newcommand{\be}{\begin{equation}}
\newcommand{\ee}{\end{equation}}
\newcommand{\bal}{\begin{align}}
\newcommand{\eal}{\end{align}}
\newcommand{\balnonum}{\begin{align*}}
\newcommand{\ealnonum}{\end{align*}}

\def\qed{\hfill \vrule height1.3ex width1.2ex depth-0.1ex}

\newtheorem{theorem}{\noindent Theorem}

\newtheorem{proposition}{\noindent Proposition}

\newtheorem{corollary}{\noindent Corollary}

\newtheorem{lemma}{\noindent Lemma}

\newtheorem{definition}{\noindent Definition}

\newenvironment{proof}{\noindent{\bf Proof:}}

\newcommand{\dominate}{\succcurlyeq}

\newcommand{\sdominate}{\succ}

\newcommand{\maxset}[2]{\mathcal{P}_{#1,#2}}
\newcommand{\weakmaxset}[2]{\mathcal{Q}_{#1,#2}}

\title{A phase transition for the probability of being a maximum among random vectors with general iid coordinates} 
\date{\today}

\author{Royi Jacobovic\footnote{This author was supported by the GIF Grant 1489-304.6/2019.} \: and Or Zuk}

\affil{Dept. of Statistics and Data Science, the Hebrew University of Jerusalem}

\begin{document}

\maketitle

\abstract{Consider $n$ iid real-valued random vectors of size $k$ having iid coordinates with a general distribution function $F$. A vector is a maximum if and only if there is no other vector in the sample that weakly dominates it in all coordinates. Let $p_{k,n}$ be the probability that the first vector is a maximum. The main result of the present paper is that if $k\equiv k_n$ grows at a slower (faster) rate than a certain factor of $\log(n)$, then $p_{k,n} \rightarrow 0$ (resp. $p_{k,n}\rightarrow1$) as $n\to\infty$. Furthermore, the factor is fully characterized as a functional of $F$. We also study the effect of $F$ on $p_{k,n}$, showing that while $p_{k,n}$ may be highly affected by the choice of $F$, the phase transition is the same for all distribution functions up to a constant factor.}

\section{Introduction}
Consider a model with a sample of $n$ iid random vectors of size $k$. It is assumed that the coordinates are iid real-valued random variables having a general distribution function $F$. A vector is said to be a (strong) maximum if and only if (iff) there is no other vector in the sample that (weakly) dominates it in all coordinates. Let $p_{k,n}$ be the probability that the first vector is a maximum. Once $k$ (resp. $n$) is fixed, then $p_{k,n}\rightarrow0$ (resp. $p_{k,n}\rightarrow1$) as $n\to\infty$ (resp. $k\to\infty$). The main contribution of the present work is a generalization of this straightforward observation by allowing $k$ to be determined as a function of $n$. Namely, we will show that if $k\equiv k_n$ grows at a slower (resp. faster) rate than $\gamma \log(n)$, then $p_{k,n}\rightarrow0$ (resp. $p_{k,n}\rightarrow1$) as $n\to\infty$, where $\gamma \in (0,1]$ is a certain constant that depends on the distribution $F$. The derivation of this result uses extreme value theory, and in particular relies on a result of Ferguson \cite{Ferguson1993} about the asymptotic behaviour of a maximum of an iid sequence of geometric random variables.

The asymptotic behaviour of $p_{k,n}$ has an important role in many applications. For example, in the analysis of linear programming \cite{blair1986random} and of maxima-finding algorithms \cite{chen2012maxima, devroye1999note,dyer1998dominance,golin1994provably,tsai2003efficient}. Furthermore, it is also related to game theory \cite{o1981number} and the analysis of random forest algorithms \cite{biau2016random,scornet2015consistency}. This literature focuses mainly on asymptotic results once $F$ is a continuous function, $k$ is fixed and $n$ tends to infinity \cite{bai1998variance,bai2005maxima,barbour2001number,barndorff1966distribution,baryshnikov2000supporting,hwang2004phase,o1981number}. Both \cite{o1981number} and \cite{barndorff1966distribution} contain an approximation of the expected number of maxima. In addition, an approximation of the variance of the number of maxima is given in \cite{bai1998variance} and asymptotic normality of this number was proved in \cite{bai2005maxima}. 

To the best of our knowledge, the only paper that includes asymptotic results as $n\to\infty$ and $k$ is determined as a function of $n$ is \cite{hwang2004phase}. 
In the last equation of Section 1.1 of \cite{hwang2004phase} there is a first order approximation of $p_{k,n}$. This approximation holds uniformly for all possible forms of variations of $k$ as a function of $n$, as $n \to \infty$. In particular, it yields existence of a non-trivial phase-transition at $k\approx\log(n)$ which is consistent with our findings. While \cite{hwang2004phase} refers only to a continuous $F$, the current results hold for a general $F$.

The rest is organized as follows: Section \ref{sec: main result} contains a precise description of the model with a statement of the main result. In particular, the functional $\gamma$ of $F$ that determines the localization of the phase transition is presented (with the proof deferred to Section \ref{sec: proof}). Section \ref{sec: examples} is devoted to exploring the effect of the distribution $F$ on the probability $p_{k,n}$, with two important special cases: Section \ref{subsec: continuous distribution} is about the continuous case and includes a detailed discussion of the relation between the current results and the approximation that appears in \cite{hwang2004phase}. Section \ref{subsec: bernoulli} is about a simple example in which the coordinates have a Bernoulli distribution. This example illustrates two points:
\begin{enumerate}
 \item While $p_{k,n}$ is the same for every continuous $F$,
 once the continuity assumption is relaxed changing the distribution $F$ can change drastically the first-order asymptotic behaviour of $p_{k,n}$ for fixed $k$ as $n \to \infty$. 
In contrast, when both $k,n \to \infty$, the phase-transition for $p_{k,n}$ is the same up to a multiplicative factor $\gamma$ for all distribution
functions $F$.
 
 \item Even for a special case in which there is a simple exact combinatorial formula for $p_{k,n}$, it is unclear how to utilize this formula in order to derive the main result directly. 
\end{enumerate}

\section{Model description and the main result}\label{sec: main result}
In the sequel, for every set $A$ and a potential element $a$, denote the corresponding indicator function
\begin{equation}
 \textbf{1}_A(a)\equiv\begin{cases} 
 1, & a\in A, \\
 0, & a\notin A.
 \end{cases}
\end{equation}
In addition, in several places of this manuscript we denote the minimum (resp. maximum) of some real numbers $x_1,x_2,\ldots,x_n$ by $\underset{i}{\wedge} x_i \equiv \min_i x_i$ (resp. $\underset{i}{\vee} x_i \equiv \max_i x_i$). In particular, when $n=2$, then we simply write $x_1\wedge x_2$ (resp. $x_1\vee x_2$). 

\subsection{Multivariate maximum}
The following is a common definition of a maximum of a set of vectors in $\mathbb{R}^k$. It is based on the product order $\preceq$ on $\mathbb{R}^k$, \textit{i.e.,} for every two vectors $a,b\in\mathbb{R}^k$ such that $a=(a_1,a_2,\ldots,a_k)$ and $b=(b_1,b_2,\ldots,b_k)$ define 
\begin{equation}
 a\preceq b\Leftrightarrow \left(a_i\leq b_i\ , \ \forall 1\leq i\leq k\right).
\end{equation}
Similarly, define 
\be
a\prec b \Leftrightarrow \left( a\preceq b \ and \ \exists i \in [k] \: s.t. \: a_i < b_i \right).
\ee

\begin{definition}\label{def: maximum}
Let $x_1,x_2,\ldots,x_n$ be $n$ vectors in $\mathbb{R}^k$. In addition, let $\preceq$ be the product order on $\mathbb{R}^k$. Then, for each $1\leq i\leq n$, $x_i$ is a maximum with respect to $x_1,x_2,\ldots,x_n$ iff there is no $j\neq i$ such that $x_i\preceq x_j$. In addition, the set of maxima with respect to $x_1,x_2,\ldots,x_n$ is called the Pareto-front generated by $x_1,x_2,\ldots,x_n$.
\end{definition}

\begin{remark}
\normalfont Definition \ref{def: maximum} refers to a \textit{strong} maximum. To see this, consider the special case in which $k=1,n\geq2$ and $x_1=x_2=\ldots=x_n$. In this case, $x_1,x_2,\ldots,x_n$ are all maxima in the usual sense but none of them is a maximum in the sense of Definition \ref{def: maximum}. \end{remark}

\begin{remark}
\normalfont It is possible to have multiple maxima in the sense of Definition \ref{def: maximum}. For instance, assume that $n=k=2$ and consider the case in which $x_1=(1,0)$ and $x_2=(0,1)$. 

\end{remark}
\begin{remark}\normalfont
\label{remark:weak_strong}
It is natural to introduce another notion of multivariate maximum: $x_i$ is a {\it weak} maximum with respect to $x_1,x_2,\ldots,x_n$ iff there is no $j\neq i$ such that $x_i\prec x_j$. Correspondingly, the set of weak maxima with respect to $x_1,x_2,\ldots,x_n$ is called the \textit{weak} Pareto-front generated by $x_1,x_2,\ldots,x_n$. Later, in Section \ref{subsec: bernoulli} we discuss this notion once the coordinates have a Bernoulli distribution. 
\end{remark}

\subsection{Problem description}
Let $\left\{X_{ij} ;i,j\geq1\right\}$ be an infinite array of iid real-valued random variables having a distribution function $F$. For every $i,k\geq1$ denote $X_i^{k}\equiv\left(X_{i1},\ldots,X_{ik}\right)$ and for every $k,n\geq1$, 
let $\maxset{k}{n} \subset \{1,2,\ldots,n\}$ be the random set of indices of all vectors that belong to the Pareto-front generated by the random vectors $X_1^k,X_2^k,\ldots,X_n^k$. 
Also, for every event $B$, denote the complement by $\overline{B}$. Then, observe that $\textbf{1}_{\maxset{k}{n}}(1)$ is equal to one iff the event 
\begin{equation}
 A_{k,n}\equiv \bigcap\limits_{j=2}^n \overline{\left\{X_1^k\preceq X_j^k\right\}} \: \end{equation}
occurs. Moreover, note that for every  sequence $(k_n)_{n\geq1}$ of positive integers, $\textbf{1}_{\maxset{k_n}{n}}(1)\xrightarrow{n\to\infty}1$ (resp. $\textbf{1}_{\maxset{k_n}{n}}(1)\xrightarrow{n\to\infty}0$) $P$-a.s. iff 
\begin{equation}
    P\left(\liminf_{n\to\infty} A_{k_n,n}\right)=1\ \ , \ \ \left[\text{resp. } P\left(\limsup_{n\to\infty} A_{k_n,n}\right)=0\right].
\end{equation}
An initial observation is that:
\begin{enumerate}
 \item For every fixed $k\geq1$, $\textbf{1}_{\maxset{k}{n}}(1)\xrightarrow{n\to\infty}0$ , $P$-a.s.
 
 \item For every fixed $n\geq1$, $\textbf{1}_{\maxset{k}{n}}(1)\xrightarrow{k\to\infty}1$ , $P$-a.s.
\end{enumerate}
The main question is how to generalize this observation by characterizing the asymptotic behaviour of $\textbf{1}_{\{\maxset{k_n}{n}\}}(1)$ as $n\to\infty$ for a general sequence $(k_n)_{n=1}^\infty$? 

\subsection{Main result}
Let $X$ be a random variable with a cumulative distribution function $F$.
Define the function $S:\mathbb{R}\rightarrow[0,1]$ as $S(x)\equiv P(X \geq x)$. When $F$ is continuous, $S$ is the corresponding survival function. Next, define 
\be
\gamma \equiv \gamma_F \equiv -E\log\left[S(X)\right].
\label{eq:gamma_def}
\ee
The following theorem is the main result. Its proof is given in Section \ref{sec: proof}. 

\begin{theorem}\label{thm: main result1}

Let $k_1,k_2,\ldots$ be a sequence of positive integers \begin{description}
\item[(a)] If
\be\label{eq: condition1}
\liminf_{n\to\infty}\frac{k_n}{\log(n)}>\gamma^{-1}\,,
\ee 
then 
\be
\textbf{1}_{\maxset{k_n}{n}}(1)\xrightarrow{n\to\infty}1\, \ \ P\text{-a.s.}
\ee
		
\item[(b)] If 
\be\label{eq: condition2}
\limsup_{n\to\infty}\frac{k_n}{\log(n)}< \gamma^{-1}\,,
\ee 
then
\be
\textbf{1}_{\maxset{k_n}{n}}(1)\xrightarrow{n\to\infty}0, \ \ P\text{-a.s.}
\ee
\end{description}
\end{theorem}
For every $k,n\geq1$, denote 
\be
p_{k,n}\equiv P\left(A_{k,n}\right)=E\textbf{1}_{\maxset{k}{n}}(1).
\label{eq:p_k_n_def}
\ee
Then, an application of bounded convergence theorem yields the following corollary.

\begin{corollary}\label{cor: main corollary}
Let $k_1,k_2,\ldots$ be a sequence of positive integers.
\begin{description}
\item[(a')] 
\be\label{eq: condition first part}
\liminf_{n\to\infty}\frac{k_n}{\log(n)}>\gamma^{-1} \Rightarrow \lim_{n \to \infty} p_{k_n,n} = 1.
\ee

\item[(b')]
\be\label{eq: condition second part}
\limsup_{n\to\infty} \frac{k_n}{\log(n)}<\gamma^{-1} \Rightarrow \lim_{n \to \infty} p_{k_n,n} = 0.
\ee
\end{description} 
\end{corollary}

\subsection{The factor $\gamma$}
\label{sec: gamma}
Define
\begin{equation}
S^{-1}(y)\equiv\inf\left\{x\in\mathbb{R};S(x)\leq y\right\}\ \ , \ \ y\in(0,1).
\end{equation}
Since $S$ is a nonincreasing leftcontinuous function, $S\left[S^{-1}(y)\right]\leq y$ for every $y\in(0,1)$.
By definition, $-\log\left[S(X)\right]\geq0$ and hence 
is well-defined and nonnegative. Furthermore, the usual formula for an expectation of a nonnegative random variable yields that
\begin{align}
\gamma &=\int_0^\infty P\left[-\log\left[S(X)\right]>t\right]dt\\&=\int_0^\infty P\left[S(X)<e^{-t}\right]dt\nonumber\\&=\int_0^\infty P\left[X>S^{-1}\left(e^{-t}\right)\right]dt\nonumber\\&=\int_0^\infty S\left[S^{-1}\left(e^{-t}\right)\right]dt\nonumber\\&\leq \int_0^\infty e^{-t}dt=1.\nonumber
\end{align}
When $F$ is continuous, the last inequality above holds with equality and $\gamma=1$.
Moreover, $\gamma=0$ if and only if $S\equiv1$, which means that $X$ is infinite. Thus, the assumption that $X$ is real-valued implies that $\gamma\in(0,1]$.
For example, when the coordinates have a $\text{Bernoulli}(p)$ distribution for some $p\in(0,1)$, 
\be
 S(x)=\begin{dcases}
 1, & x \leq 0, \\
 p, & 0 < x \leq 1,\\ 
 0, & 1 < x. \\
 \end{dcases}
\ee
Therefore, 
\begin{equation}
 \gamma= - p \log(S(1)) - (1-p)\log(S(0)) = - p \log(p)
\end{equation}
and hence $\gamma=e^{-1} \approx 0.368$ is the maximal value of $\gamma$ for the Bernoulli case, obtained at $p = e^{-1}$.

\section{The effect of the distribution $F$}
\label{sec: examples}
In this section we study the effect of the distribution $F$ of the individual variables $X_{ij}$, on the distribution of the number of maxima.
We specify the dependence on $F$ explicitly, denoting $\maxset{k}{n}^{(F)}$ the (random) maximal set and $p_{k,n}^{(F)}$ the probability of being a maxima when $X_{ij} \sim F$. Similarly, we denote by $\weakmaxset{k}{n}^{(F)}$ the weak Pareto-front generated by $X_1^k,..,X_n^k$ (see Remark \ref{remark:weak_strong}), and define 
\begin{equation}
 q_{k,n}^{(F)}\equiv P\left(1\in\weakmaxset{k}{n}^{(F)}\right)=E\textbf{1}_{\weakmaxset{k}{n}^{(F)}}(1).
\end{equation}
By definition $X_j^k \sdominate X_i^k \Rightarrow X_j^k \dominate X_i^k$, 
hence $\maxset{k}{n}^{(F)} \subseteq \weakmaxset{k}{n}^{(F)}$ and $p_{k,n}^{(F)} \leq q_{k,n}^{(F)}$. In particular, when $F$ is continuous, $\maxset{k}{n}^{(F)} = \weakmaxset{k}{n}^{(F)}$, $P$-a.s., hence $p_{k,n}^{(F)}=q_{k,n}^{(F)}$. Moreover,
let $U(\cdot)$ be the uniform distribution function on $[0,1]$ and observe that every continuous $F$ satisfies the relation
\begin{equation}
    p^{(U)}_{k,n} = p_{k,n}^{(F)} = q_{k,n}^{(F)}=q_{k,n}^{(U)}.
\end{equation}
Proposition \ref{prop:weak_strong_different_f} below shows that the continuous and the Bernoulli distributions are extreme cases, in the sense that for every distribution $F$, the probability of being a (strong) maxima lies between them. To shorten notation, for every $p\in(0,1)$, let $p_{k,n}^{(p)}$ be the probability of being a maximum once the coordinates have a $\text{Bernoulli}(p)$ distribution.

\begin{proposition}
\label{prop:weak_strong_different_f}
	 Let $F$ be a general distribution function. Then, 
	\begin{enumerate}
	 \item $p_{k,n}^{(F)} \leq p^{(U)}_{k,n}$.
	 
	 \item $p_{k,n}^{(p)} \leq p_{k,n}^{(F)}$ for every $p \in \left\{1-F(x);x\in\mathbb{R}\right\}$. 
	\end{enumerate} 
	\end{proposition}	

	\begin{proof}
	\normalfont	$\text{}$
	\begin{enumerate}
	 \item  The random variables $X_{ij} \sim F$ can be realized by taking uniform random variables $U_{ij} \sim U$, and then taking the transformation $X_{ij} = F^{-1}(U_{ij})$, where $F^{-1}$ is the pseudo-inverse of $F$. Thus, since $F^{-1}$ is nondecreasing we have $U_{j}^k \dominate U_{i}^k \Rightarrow X_{j}^k \dominate X_{i}^k $ and hence
	$X_{i}^k \in \maxset{k}{n}^{(F)} \Rightarrow U_{i}^k \in \maxset{k}{n}^{(U)}$. Therefore, $\maxset{k}{n}^{(F)} \subseteq \maxset{k}{n}^{(U)}$ and hence $p_{k,n}^{(F)} \leq p^{(U)}_{k,n}$. 
	
	\item Take $x$ with $p \equiv 1-F(x)$ and define $B_{ij} = \textbf{1}_{\{X_{ij} > x\}}$. Since $B_{ij}$ is a nondecreasing transformation of $X_{ij}$, then $B_{i}^j \in \maxset{k}{n}^{(p)} \Rightarrow X_{i}^j \in \maxset{k}{n}^{(F)}$. As a result, $\maxset{k}{n}^{(p)} \subseteq \maxset{k}{n}^{(F)}$ and hence $p_{k,n}^{(p)} \leq p_{k,n}^{(F)}$. \qed
	\end{enumerate}
	\end{proof}

\begin{remark}\normalfont
While $p_{k,n}^{(F)} \leq p^{(U)}_{k,n}$ for any $F$ (i.e. discretization may only reduce the probability of being a {\it strong} maximum), 
there is no general ordering that always holds between $q_{k,n}^{(F)}$ and $q_{k,n}^{(U)}$. This is demonstrated numerically for the Bernoulli distribution in Section
\ref{subsec: numerics}.
\end{remark}

Since the values $p_{k,n}^{(F)}$ for every distribution $F$ of the $X_{ij}$'s can be bounded by the values for the continuous and Bernoulli case, we compare these two cases to study the effect of quantization on the probability of a random vector being a maximum.

\subsection{Continuous distribution}\label{subsec: continuous distribution}

For every $k,n\geq1$, there are well-known exact formulas for $p_{k,n}^{(U)}$ (see e.g. \cite{bai2005maxima}):
\begin{enumerate}
\item 
\be 
\label{eq: combinatorial formula}
p_{k,n}^{(U)} = \sum_{u=1}^n\binom{n-1}{u-1}\frac{(-1)^{u-1}}{u^k}.
\ee
	
	\item \begin{equation}\label{eq: recurrence}
	p_{k,n}^{(U)}=\begin{dcases}
	\frac{1}{n} \sum_{u=1}^n p_{k-1,u}^{(U)}, &k>1, \\
	\frac{1}{n}, &k=1,
	\end{dcases}
	\end{equation}
	and hence, for every $k>1$ one has
	\be
	p_{k,n}^{(U)}=\frac{1}{n} \sum_{u\in\mathcal{U}_{k,n}} \frac{1}{u_1 u_2 ... u_{k-1}}
	\label{eq:sum_u_k_formula}
	\ee
	where
	\be
	\label{eq: P(S)}
	\mathcal{U}_{k,n}\equiv\left\{u=(u_1,\ldots,u_{k-1})\in\mathbb{Z}^{k-1}\ ;\ 1\leq u_1\leq u_2\leq\ldots\leq u_{k-1}\leq n\right\}.
	\ee
\end{enumerate}
Furthermore, it is well known (see, e.g., \cite{barndorff1966distribution}) that for every fixed $k$, 
\be
p_{k,n}^{(U)}\sim\frac{\log^{k-1}(n)}{n (k-1)!}\ \ \text{as}\ \ n\to\infty.
\label{eq:continuous_fixed_k_asymptotics}
\ee
For a fixed $k$, other asymptotic results regarding the size of the Pareto-front as $n\to\infty$ include asymptotic formulas for the variance \cite{bai1998variance} and a corresponding central limit theorem \cite{bai2005maxima}. 

Hwang \cite{hwang2004phase} applied analytic techniques (see, \cite{hwang1998repartition},\cite{hwang1998poisson}) to these identities in order to derive an approximation of $p_{k,n}^{(U)}$ as $n\to\infty$ and $k$ is determined as a function of $n$. Specifically, let $\Phi(\cdot)$ be the cumulative distribution function of a standard normal random variable, and let $\Gamma(\cdot)$ be the Gamma function. Then, the first order approximation which appears in \cite{hwang2004phase} is 

\begin{equation}\label{eq: Hwang2004}
 p_{k,n}^{(U)}\sim\begin{cases} 
 \frac{\log^{k-1}(n)}{n(k-1)!}\Gamma\left[1-\frac{k}{\log(n)}\right], & \log(n)-k\gg\sqrt{\log(n)}, \\
 \Phi\left[\frac{k-\log(n)}{\sqrt{\log(n)}}\right], & |k-\log(n)|=o\left[\log^{\frac{2}{3}}(n)\right], \\
 1, & \sqrt{\log(n)}\ll k-\log(n), 
\end{cases}
\end{equation}
and it holds uniformly for all variations of $k$ as $n\to\infty$. Since $\gamma=1$ for every continuous $F$, it may be verified that \eqref{eq: Hwang2004} implies Corollary \ref{cor: main corollary}. However, since convergence in $P$ does not imply convergence $P$-a.s., it is not straightforward to deduce Theorem \ref{thm: main result1} from \eqref{eq: Hwang2004}, even for the continuous case. In fact, Hwang \cite{hwang2004phase} put forth the question of whether exists a probabilistic explanation for the phase-transition at $k\approx\log(n)$? Theorem \ref{thm: main result1} yields some probabilistic explanation for this phenomenon, although it does not supply a probabilistic proof of \eqref{eq: Hwang2004}.

\subsection{Bernoulli distribution}\label{subsec: bernoulli}
In this part, we present an example of a distribution function $F$ for which it is possible to derive an explicit combinatorial expression of $p_{k,n}$. As to be shown, even when such an expression is available, still it is unclear how Theorem \ref{thm: main result1} may be deduced from it (for this special case). Furthermore, this example demonstrates the possible differences between the model with a continuous $F$ versus discontinuous $F$.

Let $X_{ij}\sim\text{Bernoulli}\left(p\right)$ for some $p \in (0,1)$. Let $B_1=\sum_{j=1}^kX_{1j}\sim\text{Binom}\left(k,p\right)$ and without loss of generality assume that $X_{1j}=1$ for every $1\leq j\leq B_1$ and $X_{1j}=0$ for every $B_1+1\leq j\leq k$. By the law of total probability applied to $B_1$, 
\begin{align}
p_{k,n}^{(p)} &= \sum_{i=0}^k \binom{k}{i} p^i (1-p)^{k-i} \left[P( X^k_1 \not\preceq X^k_2 | B_1 = i)\right]^{n-1} \nonumber \\
&= \sum_{i=0}^k \binom{k}{i} p^i (1-p)^{k-i} \left[1 - P\left( \overset{i}{\underset{j=1}{\wedge}} X_{2j} = 1 \right) \right]^{n-1} \nonumber \\
&= \sum_{i=0}^k \binom{k}{i} p^i (1-p)^{k-i} \Big(1 - p^{i}\Big)^{n-1}
\label{eq:binary_p_k_n}.
\end{align}

where in the last equation above, when $i=0,n=1$ the last term $(1-p^i)^{n-1}=0^0$ is defined to be $1$. 
The asymptotic behaviour of $p_{k,n}^{(p)}$ for fixed $k$ and $n \to \infty$ follows directly from \eqref{eq:binary_p_k_n}.
Since $(1 - p^{i})^{n-1} = o\left[(1 - p^{k})^{n-1}\right]$ for all $i<k$ as $n\to\infty$, all terms in the above sum are negligible for large $n$ except for the last, giving the result 
\be
p_{k,n}^{(p)} \sim p^{k} (1-p^{k})^{n-1} \ \ \text{as}\ \ n\to\infty.
\ee 
\begin{remark} \normalfont
While \eqref{eq:binary_p_k_n} is an exact combinatorial formula for $p^{(p)}_{k,n}$, it is not straightforward to analyze the behaviour of this combinatorial formula as $n\to\infty$ when $k$ is determined as a general function of $n$. Theorem \ref{thm: main result1} gives us the asymptotic result for $p_{k,n}$ as $k,n \to \infty$ without relying on the exact expression.
\end{remark}

A similar calculation to the one in \eqref{eq:binary_p_k_n} gives the probability of a weak maximum, 
\begin{align}
q_{k,n}^{(p)} &= \sum_{i=0}^k \binom{k}{i} p^i (1-p)^{k-i} \left[P( X^k_1 \not\prec X^k_2 | B_1 = i)\right]^{n-1} \nonumber \\
&= \sum_{i=0}^k \binom{k}{i} p^i (1-p)^{k-i} \left[1 - P\left( \overset{i}{\underset{j=1}{\wedge}} X_{2j} = 1 \right) P\left( \overset{k}{\underset{j=i+1}{\vee}} X_{2j} = 1 \right) \right]^{n-1} \nonumber \\
&= \sum_{i=0}^k \binom{k}{i} p^i (1-p)^{k-i} \Big(1 - p^{i} + p^i(1-p)^{k-i}\Big)^{n-1},
\label{eq:binary_p_k_n_weak}
\end{align}
and the asymptotic result $q_{k,n}^{(p)} \to p^{k}$ for fixed $k$ as $n \to \infty$. 
\begin{remark} \normalfont
For any fixed $k$ the decay of $p^{(U)}_{k,n}=q^{(U)}_{k,n}$ is sub-linear in $n$ as $n \to \infty$ (see \eqref{eq:continuous_fixed_k_asymptotics}).
In contrast, $p_{k,n}^{(p)}$ decays to zero exponentially fast, whereas $q_{k,n}^{(p)}$ converges to a positive constant. 
The result is intuitive because for any fixed $k$ the number of possible vectors in the Bernoulli case is finite, and the vector $(1,..,1)$ (with $k$ coordinates) appears at least once $P$-a.s. as $n \to \infty$. 
A strong maximum may exist only if this vector appears at most once, an event with an exponentially small probability in $n$. Any occurrence of this vector is a weak maximum, yielding a limit positive probability not depending on $n$, $P\big(X_i^k = (1,..,1)\big)=p^k$. 
\end{remark}

For a complete treatment of the case in which the coordinates have $\text{Bernoulli}(p)$ distribution, we derive a combinatorial formula for the variance. For every $i,j\in\{0,1\}$ define
\begin{equation}
    B_{ij}\equiv\left|\left\{1\leq r\leq k; X_{1r}=i,X_{2r}=j\right\}\right|.
\end{equation}
and observe that vector $(B_{00}, B_{01}, B_{10}, B_{11})$ has a multinomial distribution, i.e., 
\be
(B_{00}, B_{01}, B_{10}, B_{11}) \sim Multinomial\Big(k, \big((1-p)^2, p(1-p), p(1-p), p^2\big)\Big).
\label{eq:multinom}
\ee
By conditioning on this random vector deduce that
\begin{align}
& E \textbf{1}_{\{1,2\in \maxset{k_n}{n}^{(p)}\}} \nonumber \\ 
&= \sum\limits_{\substack{a,d\geq 0 \, ; \, b,c > 0 : \\ a+b+c+d=k}} \binom{k}{a, b, c, d} \left[P( X^k_1, X^k_2 \not\preceq X^k_3 | B_{00} = a, B_{01} = b, B_{10} = c, B_{11} = d)\right]^{n-2} \nonumber \\
&= \sum\limits_{\substack{a,d \geq 0 \:;\: b,c \geq 1 : \\ a+b+c+d=k}} \binom{k}{a ,b ,c ,d} \Bigg[1 - P\Big( \Big\{ \big\{ \overset{a+b}{\underset{j=a+1}{\wedge}} X_{3j} = 1 \big\} \bigcup \big\{ \overset{k-d}{\underset{j=a+b+1}{\wedge}} X_{3j} = 1 \big\} \Big\} \nonumber \\
&\bigcap \Big\{ \overset{k}{\underset{j=k-d+1}{\wedge}} X_{3j} = 1 \Big\} \Big) \Bigg]^{n-2} \nonumber \\
&= \sum\limits_{\substack{a,d\geq 0 \:;\: b,c \geq 1 : \\ a+b+c+d=k}} \binom{k}{a ,b ,c ,d} \left[1 - p^d (p^b + p^{c} - p^{b+c}) \right]^{n-2}\,,
\end{align}
hence the variance is given by:
\be
V_{k,n}^{(p)} \equiv Var(|\maxset{k}{n}^{(p)}|) = n p_{k,n}^{(p)} (1-p_{k,n}^{(p)}) + n(n-1) [E \textbf{1}_{\left\{1,2\in \maxset{k_n}{n}^{(p)}\right\}} - {p_{k,n}^{(p)}}^2].
\label{eq:var_discrete}
\ee

\begin{remark} \normalfont
When $k$ is fixed and $n \to \infty$, both the expectation $np_{k,n}^{(p)}$ and variance $V_{k,n}^{(p)}$ approach to zero as $n \to \infty$, hence the limiting distribution of the Pareto-front size is degenerate. 
An interesting question for future work is whether there exists a sequence $k=k_n$ such that the limiting distribution of the Pareto-front size $|\maxset{k}{n}^{(p)}|$ is non-degenerate.
\end{remark}

\begin{remark}
\normalfont In this part we have analyzed the relatively simple case when the underlying distribution is Bernoulli. Naturally, a follow-up question is about studying other distributions with the goal of comparing between the results. 
\end{remark}

\subsection{Numerical Results}\label{subsec: numerics}
A numerical comparison between the Bernoulli and continuous cases is shown in Figure \ref{fig:discrete_phase_transition}. The difference in the asymptotic behaviour between $p_{k,n}^{(p)}, q_{k,n}^{(p)}$ and $p^{(U)}_{k,n}(=q^{(U)}_{k,n})$ for fixed $k$ as $n \to \infty$ is shown in Figure 1.a. A numerical demonstration for the different behaviour of $p^{(U)}_{k_n,n}$ for $k_n = c \log(n)$ when $c < 1$ and $c > 1$ is shown in Figure \ref{fig:discrete_phase_transition}.b. Similarly, 
the phase transition for $\text{Bernoulli}(0.5)$ is presented in Figure \ref{fig:discrete_phase_transition}.c, illustrating the localization at $\gamma = \frac{1}{2} \log(2)$, compared to $\gamma=1$ for the continuous case.

Furthermore, as we have already shown, for fixed $k$ the asymptotic behaviours of $p_{k,n}^{(p)}$ and $q_{k,n}^{(p)}$ as $n \to \infty$ are very different. However, when both $k,n \to \infty$, Figure \ref{fig:discrete_phase_transition}.c suggests that the phase transition established by Theorem \ref{thm: main result1} for $p_{k,n}^{(p)}$ also holds for $q_{k,n}^{(p)}$. Comparing the two cases more rigorously is left for future research.

For numerical calculation of $p^{(U)}_{k,n}$ we have used the recurrence relation \eqref{eq: recurrence}, because the alternating sum in the combinatorial formula \eqref{eq: combinatorial formula} causes numerical instabilities.
As a result, computing $p^{(U)}_{k,n}$ for fixed $k$ requires $O(n)$ operations, and $p^{(U)}_{k.n}$ was calculated for values up to $n=10^7$ in Figure \ref{fig:discrete_phase_transition}.b. 
In contrast, the discrete combinatorial formula \eqref{eq:binary_p_k_n} for $p_{k,n}^{(p)}$ can be applied directly, enabling us to compute this probability for much larger values of $n$ (up to $n \approx 10^{130}$) in Figure \ref{fig:discrete_phase_transition}.c.
The code for all numeric calculations is freely available at \href{https://github.com/orzuk/Pareto}{https://github.com/orzuk/Pareto}.

\begin{figure}[H]
	\centering
	\includegraphics[width=0.327\textwidth]{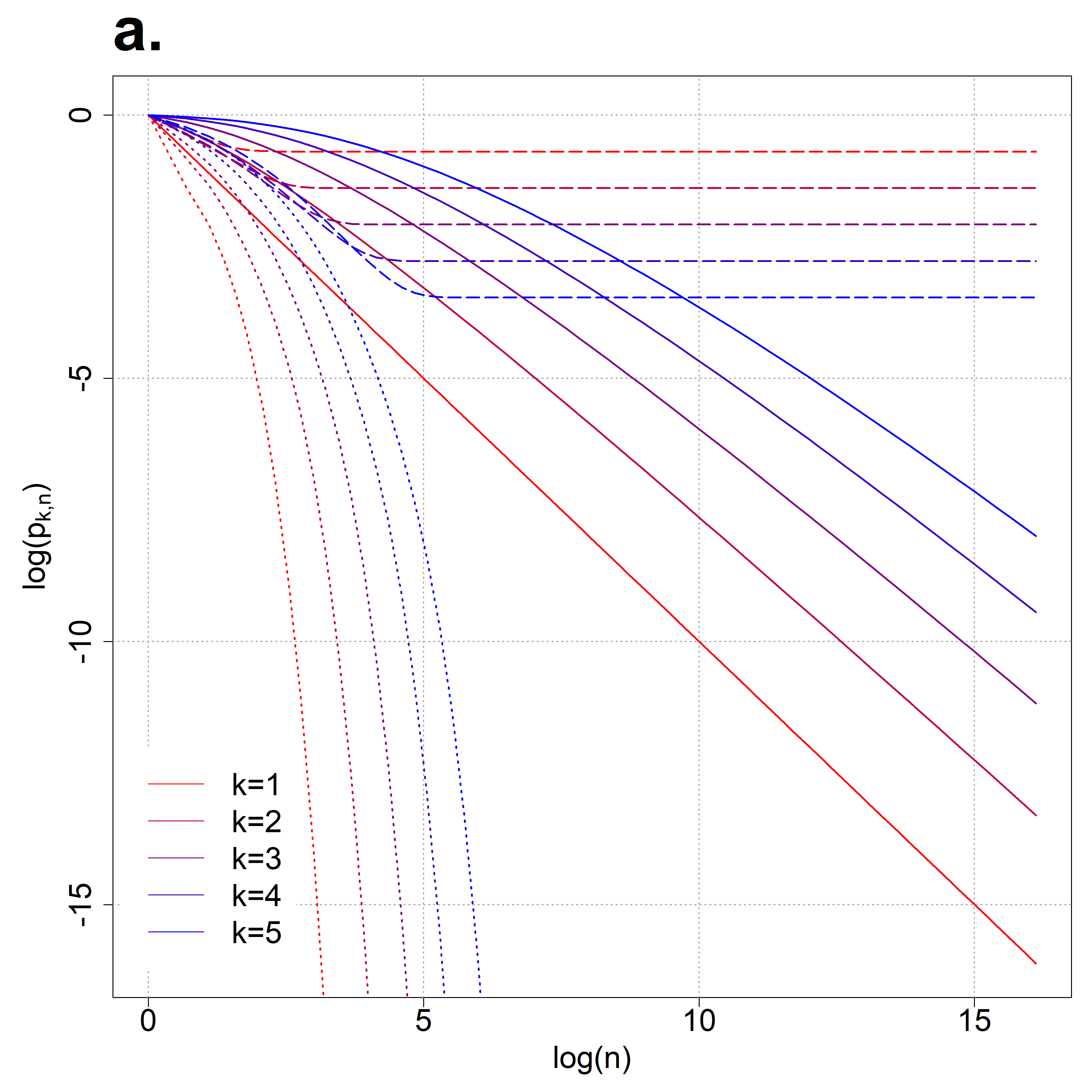} 
	\includegraphics[width=0.327\textwidth]{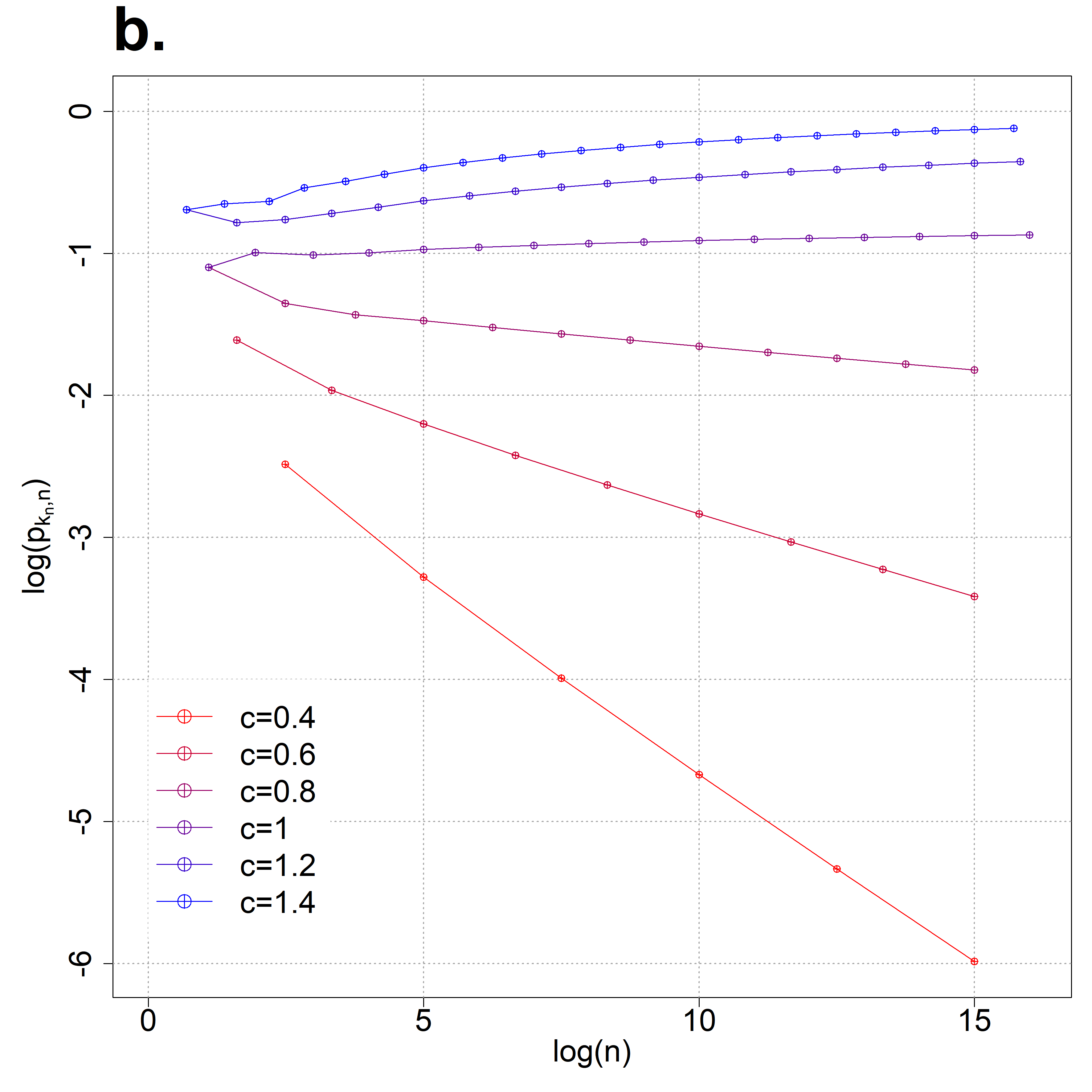}
	\includegraphics[width=0.327\textwidth]{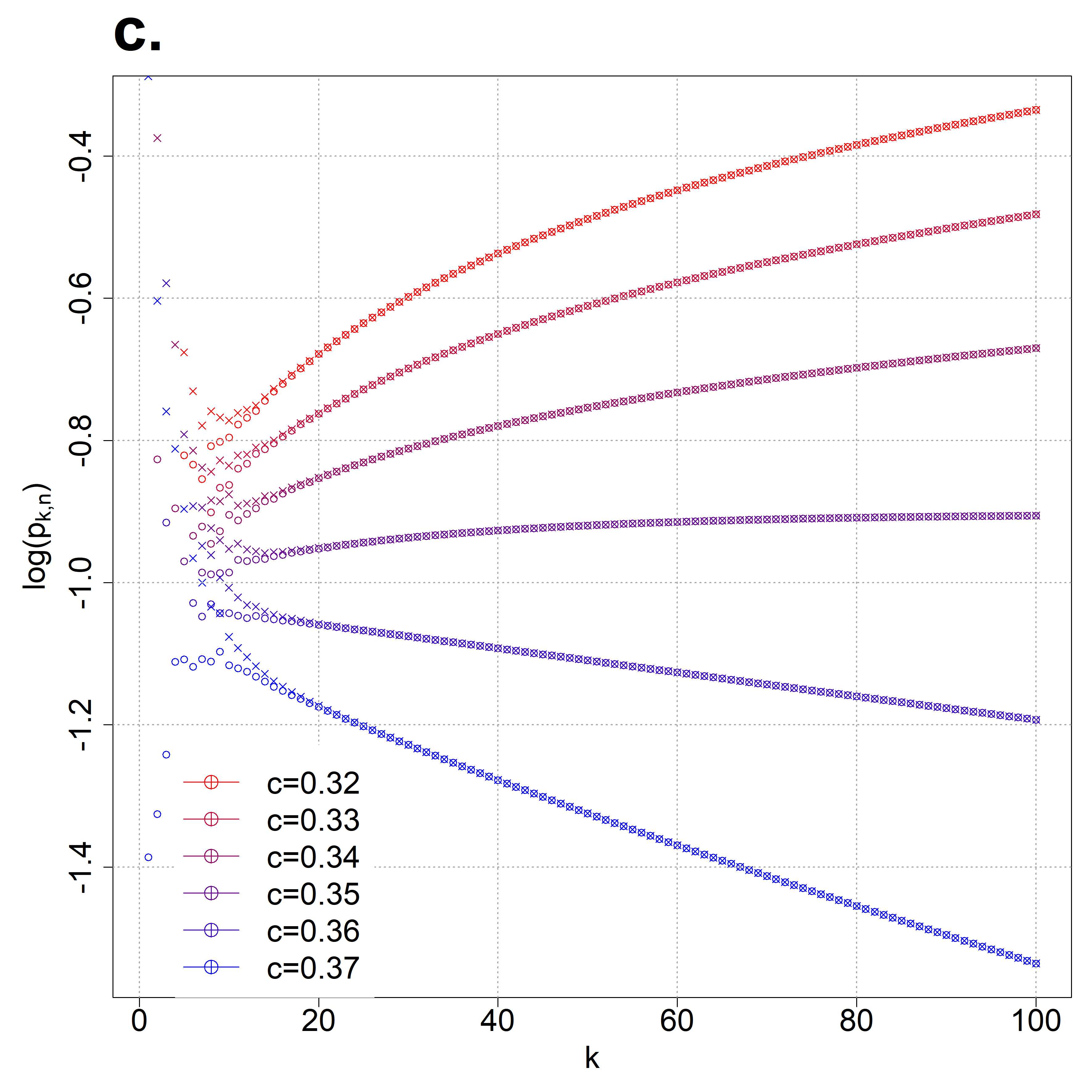}
	\caption{\footnotesize 
	 {\bf a.} Value of $p^{(U)}_{k,n}=q^{(U)}_{k,n}$ (solid lines), $q_{k,n}^{(0.5)}$ (dashed lines) and $p_{k,n}^{(0.5)}$ (dotted lines) as a function of $n$, shown on a log-scale, for $k=1,2,3,4,5$. While $p_{k,n}^{(0.5)} < p_{k,n}$ for all $k$ and $n$, when $n$ is large $q_{k,n}^{(0.5)}$ can exceed $p_{k,n}^{(U)}$. 
		{\bf b.} Value of $\log(p_{k_n,n}^{(U)})$  using the exact combinatorial formula (line-connected circles) for $k_n = \lfloor( c \log(n) \rfloor$ for $n$ from $1$ to $10^7$ and $k_n$ up to $\lfloor( c \log(10^7) \rfloor$ for each $c$. We were able to compute $p^{(U)}_{k,n}$ accurately only for small values of $k$, due to the recurrence relation in \eqref{eq: recurrence} and the alternating sum in \eqref{eq: combinatorial formula}. For $c \leq 0.8$ the curves decrease with $n$, consistent with the result that $p^{(U)}_{k_n,n} \to 0$ for this case. For $c \geq 1.2$ the curves increase towards zero with $n$ , consistent with the result that $p^{(U)}_{k_n,n} \to 1$ for this case. For $c=1$ there seems to be a slight increase in $p^{(U)}_{k_n,n}$ too, but results are inconclusive.
		{\bf c.} Value of $\log(p_{k_n,n}^{(0.5)})$ ('x' symbols) and $\log(q_{k_n,n}^{(0.5)})$ ('o' symbols) for the $\text{Bernoulli}(0.5)$ case, for $k_n = c \log(n)$ for different values of $c$. For $c < \gamma =\frac{\log(2)}{2} = 0.34657$ the log-probabilities approach $0$, whereas for $c > \gamma$ the log-probabilities decreases to $-\infty$. For all values of $c$, the ratio $\frac{q_{k_n,n}^{(0.5)}}{p_{k_n,n}^{(0.5)}}$ approaches $1$ as $n \to \infty$. 
 \label{fig:discrete_phase_transition}} 
\end{figure}

\section{Proof of Theorem \ref{thm: main result1}}
\label{sec: proof}
For every $i\geq2$, let 
\be
G_i^{1}\equiv\min\left\{k\geq1; X_{ik}< X_{1k}\right\}-1.
\ee
Then $X_1^k\preceq X_i^k$ for every $1\leq k\leq G^{1}_i$ and $X_1^k\not \preceq X_i^k$ for every $k> G^{1}_i$. In particular, this implies that for every $n,k\geq1$, 
\be\label{eq: max condition}
1 \in \maxset{k}{n}\Leftrightarrow M_n^{1}\equiv\max_{2\leq i\leq n} G_i^1\leq k-1
\ee
with the convention that a maximum over an empty-set of numbers equals zero. Thus, the asymptotic behaviour of $\textbf{1}_{\maxset{k}{n}}(1)$ as $n,k\to\infty$ is strongly related to the asymptotic behaviour of $M_n^1$ as $n\to\infty$. Observe that $M_n^1$ is a maximum of $n-1$ identically distributed {\it dependent} geometric random variables $G_2^1,G_3^1,\ldots, G_n^1$ having a success probability $P\left(X_{11}>X_{21}\right)$. The following lemma couples $M_n^1$ with a maximum of $n-1$ {\it independent} geometric random variables.

\begin{lemma}\label{proposition: basic lemma}
	 Let $\left\{X_{ij};i,j\geq1\right\}$ be iid random variables with a distribution function $F$ for which $\gamma\equiv\gamma_F$ as defined in \eqref{eq:gamma_def}. In addition, let $G_2,G_3,\ldots$ be an iid sequence of geometric random variables with success probability $\alpha\in(0,1)$. 
	For every $n\geq1$, denote $M_n\equiv M_n^{(\alpha)}\equiv \underset{2\leq i\leq n}{\max} G_i$, and assume that $\left\{G_i ; i\geq2\right\}$ and $\left\{X_{ij};i,j\geq1\right\}$ are independent. 
	
	Then, for every $\alpha\in(0,1)\setminus\{1-e^{-\gamma}\}$, the random variable
	\begin{equation}
	 N_\alpha\equiv1+\begin{cases} 
      \sup\left\{n\geq1;M_n>M_n^1\right\}\vee0 & 1-\alpha<e^{-\gamma} \\
     \sup\left\{n\geq1;M_n<M_n^1\right\}\vee0 & 1-\alpha>e^{-\gamma} 
   \end{cases}   
	\end{equation}
	is $P$-a.s. finite.
\end{lemma}

\begin{remark}\normalfont
By definition, whenever $1-\alpha<e^{-\gamma}$ (resp. $1-\alpha<e^{-\gamma}$), then $M_n\leq M_n^1$ (resp. $M_n\geq M_n^1$) for every $n\geq N_\alpha$. 
\end{remark}

\begin{proof}
	For every $k\geq1$ denote 
	\be\label{eq: tau def}
	\tau_k^1\equiv\min\left\{i\geq2;M_i^1\geq k\right\} , \ \ \tau_k \equiv \min\left\{i\geq2;M_i \geq k\right\}.
	\ee
	Conditioned on $X_1\equiv\left(X_{1j}\right)_{j=1}^\infty$, the events
	\begin{equation}
	\left\{X_1^k\preceq X_i^k\right\} , \ \ i\geq2
	\end{equation}
	are independent. Therefore, the random variables $\tau_k^1$ and $\tau_k$ are conditionally independent given $X_1$, such that (notice that the index $i$ in \eqref{eq: tau def} is not less than 2)
	\begin{equation}
	\tau_k^1 | X_1-1 \sim\text{Geo}\left(\prod_{j=1}^k S(X_{1j})\right)
	\end{equation}
	and
	\begin{equation}
	\tau_k-1 \sim \text{Geo}\left(\left(1-\alpha\right)^k\right).
	\end{equation} 
	In addition, as explained in Section \ref{sec: gamma},
	$S\left(X_{11}\right),S\left(X_{12}\right),\ldots$ are iid random variables and $-E \log S(X_{11})=\gamma\in(0,1]$. Therefore, by the strong law of large numbers
	\be\label{eq: SLLN}
	L_k\equiv\frac{1}{k}\sum_{j=1}^k\left[-\log S(X_{1j})\right]\xrightarrow{k\to\infty}\gamma , \ \ P\text{-a.s.}
	\ee
	and it follows that $e^{L_k} \xrightarrow{k\to\infty} e^\gamma$, $P$-a.s. and $e^{-kL_k}\xrightarrow{k\to\infty}0$, $P$-a.s.
	
	Consider the case where $1-\alpha<e^{-\gamma}$. Then, \eqref{eq: SLLN} implies that there exists a $P$-a.s. finite random variable $K_\alpha$ which is uniquely determined by $X_1$ such that for every $k>K_\alpha$
	\be
	(1-\alpha) e^{L_k} \leq \frac{1+(1-\alpha) e^\gamma}{2} \equiv \zeta_\alpha 
	\ee
	such that $\zeta_\alpha<1$.
	In addition, $e^{-k L_k} \leq 1$ for every $k \geq 1$. Therefore, by a well-known result about a minimum of two independent geometric random variables, deduce that
	
	\begin{align}
	\sum_{k=K_\alpha}^\infty P\left(\tau_k \leq \tau_k^1\big|X_1\right)&=\sum_{k=K_\alpha}^\infty P\left(\tau_k-1 \leq \tau_k^1-1\big|X_1\right)\\
	&=\sum_{k=K_\alpha}^\infty\frac{(1-\alpha)^k}{(1-\alpha)^k+e^{-k L_k}-(1-\alpha)^k e^{-k L_k}} \nonumber\\
	& \leq \sum_{k=K_\alpha}^\infty \left[(1-\alpha)e^{L_k}\right]^k 
	\nonumber\\
	&\leq\sum_{k=K_\alpha}^\infty\zeta_\alpha^k<\infty\nonumber.
	\end{align}
	Thus, the Lemma of Borel-Cantelli implies that
	\be
	P\left(\tau_k \leq \tau_k^1\ , \ \text{i.o}\ \big| X_1\right)=0 , \ \ P\text{-a.s.}
	\ee
	and hence 
	\be
	P\left(\tau_k \leq\tau_k^1\ , \ \text{i.o}\ \right)=E\left[P\left(\tau_k \leq \tau_k^1\ , \ \text{i.o}\ \big| X_1\right)\right]=0.
	\ee
	Therefore, $P(M_n > M_n^1 \, \ \text{i.o}) = 0$,
	which yields the required result when $1-\alpha<e^{-\gamma}$.

	Assume that $1-\alpha>e^{-\gamma}$.
	Then, applying similar arguments to those that appear above yields the existence of a P-a.s. finite random variable $K_{\alpha}$ such that for any $k>K_\alpha$:
	\be
	(1-\alpha) e^{L_k} \geq \frac{1+(1-\alpha) e^\gamma}{2} \equiv \zeta_\alpha 
	\ee
	such that $\zeta_\alpha>1$. In addition, for every $k\geq1$, $(1-\alpha)^k\leq1$ and hence 
	\begin{align}
	\sum_{k=K_\alpha}^\infty P\left(\tau_k \geq \tau_k^1\big|X_1\right)&=\sum_{k=K_\alpha}^\infty P\left(\tau_k-1 \geq \tau_k^1-1\big|X_1\right)\\
	&=\sum_{k=K_\alpha}^\infty\frac{ e^{-k L_k}}{(1-\alpha)^k+e^{-k L_k}-(1-\alpha)^k e^{-k L_k}} \nonumber \\
	&\leq\sum_{k=K_\alpha}^\infty \left[(1-\alpha)e^{L_k}\right]^{-k} 
	\nonumber\\&
	\leq\sum_{k=K_\alpha}^\infty\zeta_\alpha^{-k}<\infty.\nonumber
	\end{align}
	Thus, the claim follows from the Lemma of Borel-Cantelli using a similar argument as in the previous case. \qed \newline
\end{proof}
\subsubsection*{Proof of Theorem \ref{thm: main result1} (continuation)}
It is possible to use Lemma \ref{proposition: basic lemma} in order to show that 
\be\label{eq: extended ferguson}
\frac{M_n^1}{\log(n)}\xrightarrow{n\to\infty}\gamma^{-1} , \ \ P\text{-a.s.}
\ee
To this end, fix $\epsilon>0$ and let $0<\alpha_1,\alpha_2<1$ be such that
\begin{equation*}
(1-\alpha_1)e^\gamma<1 < (1-\alpha_2)e^\gamma
\end{equation*}
and
\be
\left|\gamma^{-1}- \left[\log(1-\alpha_l)\right]^{-1}\right|<\frac{\epsilon}{2}  , \ \ \forall l=1,2.
\ee
Consider two independent iid sequences $G_2^{(\alpha_1)},G_3^{(\alpha_1)},\ldots$ and $G_2^{(\alpha_2)},G_3^{(\alpha_2)},\ldots$ such that $G_1^{(\alpha_l)}\sim\text{Geo}(\alpha_l)$ for $l=1,2$. Respectively, define the corresponding sequences of partial maxima
\begin{equation}
 M_n^{(\alpha_l)}\equiv\max_{2\leq i\leq n}G_i^{(\alpha_l)} , \ \ n\geq2\,,
\end{equation}
for each $l=1,2$ as described in the statement of Lemma \ref{proposition: basic lemma}. Then, Lemma \ref{proposition: basic lemma} implies that there exists $P$-a.s. finite random variables $N_{\alpha_1}$ and $N_{\alpha_2}$ such that
\be
M_n^{(\alpha_1)}\leq M_n^1\leq M_n^{(\alpha_2)} , \ \ \forall n\geq\max( N_{\alpha_1}, N_{\alpha_2})\equiv N.
\ee
Furthermore, Theorem $2$ of \cite{Ferguson1993} yields that for each $l=1,2$ 
\be
\frac{M_n^{(\alpha_l)}}{\log(n)}\xrightarrow{n\to\infty} - \left[\log(1-\alpha_l)\right]^{-1}\ \ \ , \ \ P\text{-a.s.}
\ee
As a result, there exists a $P$-a.s. finite random variable $N^*\geq N$ such that for every $n\geq N^*$ one has 
\be
\gamma^{-1}-\frac{\epsilon}{2}\leq\frac{M_n^{(\alpha_1)}}{\log (n)}\leq\frac{M_n^1}{\log(n)}\leq\frac{M_n^{(\alpha_2)}}{\log(n)}\leq \gamma^{-1}+\frac{\epsilon}{2}
\ee
and hence \eqref{eq: extended ferguson} follows. 
Therefore, \eqref{eq: condition1} implies that
\be\label{eq: 1}
\liminf_{n\to\infty}\frac{k_n-1}{M_n^1}=\liminf_{n\to\infty}\frac{k_n}{\log(n)}\cdot\frac{k_n-1}{k_{n}}\cdot\frac{\log(n)}{M^1_n}>1, \ \ P\text{-a.s.}
\ee
and hence \eqref{eq: max condition} yields that $\textbf{1}_{ \maxset{k_n}{n}}(1)\xrightarrow{n\to\infty}1$, $P$-a.s. Similarly, \eqref{eq: condition2} implies that
\begin{equation}
\limsup_{n\to\infty}\frac{k_n}{M_n^1}= \limsup_{n\to\infty}\frac{k_n}{\log(n)}\cdot\frac{\log(n)}{M^1_n}< 1, \ \ P\text{-a.s.} 
\end{equation}
and hence \eqref{eq: max condition} yields that $\textbf{1}_{\maxset{k_n}{n}}(1)\xrightarrow{n\to\infty}0$, $P$-a.s. 
\qed
\newline\newline
\textbf{Acknowledgement:} The authors thank the reviewer for detecting a mistake in the  original proof of Theorem \ref{thm: main result1}. 

\bibliographystyle{unsrt}
\bibliography{Pareto.bib}

\begin{thebibliography}{10}

\bibitem{Ferguson1993}
Thomas~S. Ferguson.
\newblock On the asymptotic distribution of max and mex.
\newblock {\em Statistical Papers}, 34(1):97--111, 1993.

\bibitem{blair1986random}
Charles Blair.
\newblock Random inequality constraint systems with few variables.
\newblock {\em Mathematical Programming}, 35(2):135--139, 1986.

\bibitem{chen2012maxima}
Wei-Mei Chen, Hsien-Kuei Hwang, and Tsung-Hsi Tsai.
\newblock Maxima-finding algorithms for multidimensional samples: A two-phase
  approach.
\newblock {\em Computational Geometry}, 45(1-2):33--53, 2012.

\bibitem{devroye1999note}
Luc Devroye.
\newblock A note on the expected time for finding maxima by list algorithms.
\newblock {\em Algorithmica}, 23(2):97--108, 1999.

\bibitem{dyer1998dominance}
Martin~E Dyer and John Walker.
\newblock Dominance in multi-dimensional multiple-choice knapsack problems.
\newblock {\em Asia-Pacific Journal of Operational Research}, 15(2):159, 1998.

\bibitem{golin1994provably}
Mordecai~J Golin.
\newblock A provably fast linear-expected-time maxima-finding algorithm.
\newblock {\em Algorithmica}, 11(6):501--524, 1994.

\bibitem{tsai2003efficient}
Tsung-Hsi Tsai, Hsien-Kuei Hwang, and Wei-Mei Chen.
\newblock Efficient maxima-finding algorithms for random planar samples.
\newblock {\em Discrete Mathematics \& Theoretical Computer Science}, 6, 2003.

\bibitem{o1981number}
Barry O'Neill.
\newblock The number of outcomes in the pareto-optimal set of discrete
  bargaining games.
\newblock {\em Mathematics of Operations Research}, 6(4):571--578, 1981.

\bibitem{biau2016random}
G{\'e}rard Biau and Erwan Scornet.
\newblock A random forest guided tour.
\newblock {\em Test}, 25(2):197--227, 2016.

\bibitem{scornet2015consistency}
Erwan Scornet, G{\'e}rard Biau, and Jean-Philippe Vert.
\newblock Consistency of random forests.
\newblock {\em The Annals of Statistics}, 43(4):1716--1741, 2015.

\bibitem{bai1998variance}
Zhi-Dong Bai, Chern-Ching Chao, Hsien-Kuei Hwang, and Wen-Qi Liang.
\newblock On the variance of the number of maxima in random vectors and its
  applications.
\newblock {\em The Annals of Applied Probability}, 8(3):886--895, 1998.

\bibitem{bai2005maxima}
Zhi-Dong Bai, Luc Devroye, Hsien-Kuei Hwang, and Tsung-Hsi Tsai.
\newblock Maxima in hypercubes.
\newblock {\em Random Structures \& Algorithms}, 27(3):290--309, 2005.

\bibitem{barbour2001number}
Andrew~D Barbour and A~Xia.
\newblock The number of two-dimensional maxima.
\newblock {\em Advances in Applied Probability}, 33(4):727--750, 2001.

\bibitem{barndorff1966distribution}
Ole Barndorff-Nielsen and Milton Sobel.
\newblock On the distribution of the number of admissible points in a vector
  random sample.
\newblock {\em Theory of Probability \& Its Applications}, 11(2):249--269,
  1966.

\bibitem{baryshnikov2000supporting}
Yuliy Baryshnikov.
\newblock Supporting-points processes and some of their applications.
\newblock {\em Probability Theory and Related Fields}, 117(2):163--182, 2000.

\bibitem{hwang2004phase}
Hsien-Kuei Hwang.
\newblock Phase changes in random recursive structures and algorithms.
\newblock In {\em Probability, Finance and Insurance}, pages 82--97. World
  Scientific, 2004.

\bibitem{hwang1998repartition}
Hsien-Kuei Hwang.
\newblock Sur la repartition des valeurs des fonctions arithmetiques. le nombre
  de facteurs premiers d'un entier.
\newblock {\em Journal of Number Theory}, 69(2):135--152, 1998.

\bibitem{hwang1998poisson}
Hsien-Kuei Hwang.
\newblock A poisson* geometric convolution law for the number of components in
  unlabelled combinatorial structures.
\newblock {\em Combinatorics, Probability and Computing}, 7(1):89--110, 1998.

\end{thebibliography}

\end{document}